\documentclass[11pt]{amsart}

\pdfoutput=1
\usepackage{amsfonts, amstext, amsmath, amsthm, amscd, amssymb}
\usepackage{enumerate}
\usepackage{epsfig, color, graphicx}
\usepackage{pinlabel}

\let\oldmarginpar\marginpar
\renewcommand\marginpar[1]{\oldmarginpar[\raggedleft\footnotesize #1]%
{\raggedright\footnotesize #1}}

\usepackage[
pdfborderstyle={},
pdfborder={0 0 0},
pagebackref,
pdftex]{hyperref}

\renewcommand*{\backref}[1]{}
\renewcommand*{\backrefalt}[4]{
  \ifcase #1 %
   [No citations.]%
  \or
   [#2]%
  \else
   [#2]%
  \fi
}

\newtheorem{theorem}{Theorem}[section]
\newtheorem{lemma}[theorem]{Lemma}
\newtheorem{corollary}[theorem]{Corollary}

\newtheorem{conjecture}[theorem]{Conjecture}

\newtheorem{proposition}[theorem]{Proposition}

\newtheorem*{namedtheorem}{\theoremname}
\newcommand{\theoremname}{testing}
\newenvironment{named}[1]{\renewcommand{\theoremname}{#1}\begin{namedtheorem}}{\end{namedtheorem}}

\theoremstyle{definition}
\newtheorem{define}[theorem]{Definition}
\newtheorem{example}[theorem]{Example}

\newcommand{\ZZ}{\mathbb{Z}} 

\newcommand{\QQ}{\mathbb{Q}}
\newcommand{\CC}{\mathbb{C}} 
\newcommand{\RR}{\mathbb{R}}

\newcommand{\FF}{\mathcal{F}}

\renewcommand{\setminus}{{\smallsetminus}}
\newcommand{\bdy}{\partial}
\newcommand{\lmin}{\ell_{\rm{min}}}

\def\OO{\mathcal{O}} 
\def\PP{\mathcal{P}}
\def\QQ{\mathcal{Q}}

\newcommand{\MO}{M_{\OO}} 
\newcommand{\OK}{\OO_K}
\newcommand{\OL}{\OO_L}
\newcommand{\Omin}{\OO_{\rm{min}}}

\def\vol{\mathrm{Vol}}

\def\mod{\mbox{mod}\ }

\def\HH{\mathbb{H}}

\def\QQ{\mathcal{Q}}

\def\FF{\mathcal{F}}

\begin{document}

\title{Small volume link orbifolds}
\author{Christopher K. Atkinson}
\address{Department of Mathematics, 
University of Minnesota,
Morris, MN 56267}
\email{catkinso@morris.umn.edu}

\author{David Futer}
\address{Department of Mathematics, Temple University,
Philadelphia, PA 19122}
\email{dfuter@temple.edu}
\thanks{{Futer is supported in part by NSF grant DMS--1007221.}}


\begin{abstract} 
This paper proves lower bounds on the volume of a hyperbolic
$3$--orbifold whose singular locus is a link. We identify the unique smallest
volume orbifold whose singular locus is a knot or link in the $3$--sphere, or
more generally in a $\ZZ_6$ homology sphere. We also prove more general lower
bounds under mild homological hypotheses.  
\end{abstract}

\maketitle

\section{Introduction}

The volume of a hyperbolic manifold is one of its most natural geometric invariants, and one that has ben the subject of intensive study.
The work of J{\o}rgensen and Thurston in the 1970s established that the set of volumes of 
hyperbolic $3$--manifolds is a closed, non-discrete, well-ordered subset of $\RR$. Dunbar and Meyerhoff proved that the same property holds if one considers the quotients of hyperbolic manifolds by a discrete group action, namely orbifolds \cite{dunbar-meyerhoff}. One important consequence of this work is that every infinite family of hyperbolic manifolds or orbifolds contains a finite number that realize the smallest volume.

The last decade has witnessed considerable progress on determining the volume minimizers within various families. For instance, Cao and Meyerhoff identified the two smallest non-compact, orientable hyperbolic $3$--manifolds \cite{cao-meyerhoff}. More recently, Gabai, Meyerhoff and Milley have identified the smallest closed orientable $3$--manifold  \cite{moms, milley}. Their theorem is worth stating in full, since we will use it below.

\begin{theorem}[Gabai--Meyerhoff--Milley]\label{thm:weeks}
The Weeks manifold $M_W$ is the unique closed, orientable, hyperbolic $3$--manifold
of minimal volume.  Its volume is $\vol(M_W) = 0.9427...$.
\end{theorem}

One concrete description of the Weeks manifold $M_W$ is that it is the three-fold cyclic branched cover of $S^3$, branched over the $5_2$ knot depicted in Figure \ref{fig:knots}, left.

In this paper, we study the corresponding problem for \emph{orbifolds}. For simplicity, we will implicitly assume that all $3$--manifolds and $3$--orbifolds mentioned are orientable.

\begin{define}\label{def:orbifold}
A $3$--dimensional (orientable) \emph{orbifold} $\OO$ is 
locally modeled on $B^3 / G$, where $B^3$ is the closed $3$--ball and $G \subset SO(3)$ is a finite group acting on $B^3$ by rotations. The underlying topological space or \emph{base space} is denoted $X_\OO$. If a point $x \in B^3$ is fixed by a non-trivial element of $G$, the quotient of this point in $X_\OO$ belongs to the \emph{singular locus} $\Sigma_\OO$.

Combinatorially, the singular locus $\Sigma_\OO$ is a graph, whose vertices in the interior of $X_\OO$ have valence $3$ and whose vertices on $\bdy X_\OO$ have valence $1$. The neighborhood of an interior point of an edge of $\Sigma_\OO$ is the quotient of $B^3$ under a cyclic group $\ZZ_n$. We call the integer $n \geq 2 $ the \emph{torsion order} of the edge, and decorate the edge with $n$. The neighborhood of a trivalent vertex of $\Sigma_\OO$ is the quotient of $B^3$ under a spherical triangle group.

A $3$--orbifold $\OO$ is called \emph{hyperbolic} if $\OO = \HH^3 / \Gamma$, where $\Gamma = \pi_1(\OO)$ is a discrete group of isometries, possibly with torsion. By analogy with the manifold setting, $\Gamma$ is called the \emph{fundamental group} of $\OO$. In this setting, the singular locus $\Sigma_\OO$ is covered by the fixed points of torsion elements of $\Gamma$.
\end{define}

The question of finding volume minimizers also makes sense within the family of hyperbolic orbifolds. In a long series of papers, culminating in \cite{gehring-martin:minimal-orbifold, marshall-martin:minimal-orbifold2}, Gehring, Marshall, and Martin have determined the unique smallest--volume (orientable) hyperbolic $3$--orbifold. Its volume is approximately $0.03905$. One consequence of this theorem is that a hyperbolic $3$--manifold with given volume $V$ must have a symmetry group $G$ of order at most $|G| \leq V / 0.039$.

In this paper, we focus on orbifolds whose singular locus contains no vertices. 

\begin{define}\label{def:link-orbifold}
A $3$--orbifold $\OO$ is called a \emph{link orbifold} if its singular locus is a link in $X_\OO$. In other words, the singular locus $\Sigma_\OO$ is a disjoint union of closed curves. 
\end{define}

If $\OO$ is a hyperbolic link orbifold, the lift of $\Sigma_\OO$ to $\HH^3$ consists of a disjoint union of hyperbolic lines. In their work, Gehring, Marshall, and Martin called these hyperbolic lines \emph{simple axes}. Estimating volume in the presence of a simple axis was a crucial and difficult component of their program to identify the lowest--volume orbifold \cite{gehring-martin:minimal-orbifold, marshall-martin:minimal-orbifold2}. They managed to show that a link orbifold must have volume at least $0.041$. Although this was sufficient for their purposes, Gehring and Martin speculated that link orbifolds should have substantially higher volume \cite[Page 124]{gehring-martin:minimal-orbifold}.


\begin{figure}
\labellist
\small\hair 2pt
\pinlabel {$3$} [br] at 10 82
\pinlabel {$2$} [r] at 116 53
\pinlabel {$3$} [br] at 140 82
\endlabellist
\scalebox{1}{\includegraphics{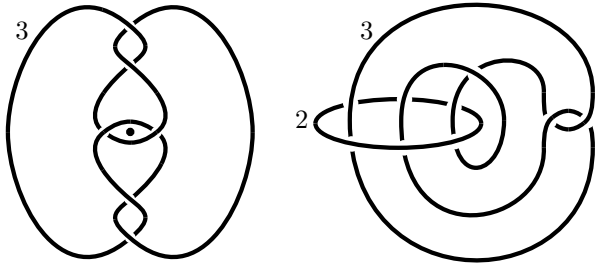}}
\caption{Left: the link orbifold $\OO_K$ with singular locus the $5_2$ knot. Right: the link orbifold $\OO_{L}$ is a two-fold quotient of $\OO_K$, by a rotation about the center. By Theorem \ref{thm:knots-hom}, $\OL$ minimizes volume among link orbifolds in $S^3$.}
\label{fig:knots}
\end{figure}

Based on computer experimentation and the results mentioned below, we conjecture that the smallest--volume link orbifold indeed has volume much larger than 0.041.

\begin{conjecture}\label{conj:minimal-link}
The link orbifold $\OL$ depicted in Figure \ref{fig:knots}, right, is the unique hyperbolic link orbifold of minimal volume. Its volume is $0.1571...$.
\end{conjecture}

Both of the orbifolds depicted in Figure \ref{fig:knots} are quotients of the Weeks manifold $M_W$. Recall that $M_W$ is a three-fold branched cover of $S^3$ branched over the $5_2$ knot, which is the singular locus of $\OK$. This is equivalent to saying that $M_W$ is a three-fold cyclic orbifold cover of $\OK$. Similarly, Figure \ref{fig:knots} illustrates that $\OK$ is a two-fold orbifold cover of $\OL$. Therefore,
$$6 \cdot \vol(\OL) \: = \: 3 \cdot \vol(\OK) \: = \: \vol(M_W) \: = \: 0.9427...$$

This paper contains several results in the direction of Conjecture \ref{conj:minimal-link}. First, we prove that $\OK$ is minimal among knot orbifolds, and $\OL$ among link orbifolds, if the base space is a homology $3$--sphere.

\begin{theorem}\label{thm:knots-hom}
Let $\OO$ be a hyperbolic  link orbifold whose base space has $H_1( X_\OO; \ZZ_6) = 0$.
\begin{enumerate}
\item If the singular locus $\Sigma_\OO$ is a knot, then $\vol(\OO) \geq 0.31423...$, with equality if and only if $\OO = \OK$.
\item If the singular locus $\Sigma_\OO$ is a link with multiple components, then $\vol(\OO) \geq 0.15711...$, with equality if
and only if $\OO = \OL$.
\end{enumerate}
\end{theorem}

In particular, $\OK$ is the smallest knot orbifold in $S^3$, and $\OL$ is the smallest link orbifold.

In fact, the hypotheses of Theorem \ref{thm:knots-hom} can be relaxed somewhat: instead of requiring that $X_\OO$ is a homology sphere, it suffices to require that each component of $\Sigma_\OO$ is  trivial in $H_1(X_\OO, \ZZ_6)$. See, for instance, Theorem \ref{thm:3null} in Section \ref{sec:null-hom}. On the other hand, some version of the homological hypotheses is crucial to our line of argument.

By relaxing the homological hypotheses even further, we obtain the following result, which is more general than Theorem \ref{thm:knots-hom} but produces a weaker volume estimate.

\begin{theorem}\label{thm:2torsion-nullhom}
If $\OO$ is a hyperbolic link orbifold such that the locus of $2$--torsion is trivial in $H_1(X_\OO,  \ZZ_2$), then
$\vol(\OO) \geq 0.1185$. If the locus of $2$--torsion is empty, then $\vol(\OO) \geq 0.2371$.
\end{theorem}

One consequence of Theorem \ref{thm:2torsion-nullhom}, combined with Theorem \ref{thm:restrictions} in Section \ref{sec:reduction}, is the following partial answer to Conjecture \ref{conj:minimal-link}.

\begin{corollary}\label{cor:minimal23}
The smallest volume hyperbolic link orbifold must have $2$--torsion, may or may not have $3$--torsion, and has no $p$--torsion for $p > 3$.
\end{corollary}

Finally, if we restrict to orbifolds without any $2$--torsion or $3$--torsion, the estimate becomes even larger, and we can also pinpoint the unique minimizer in this family.

\begin{theorem}\label{thm:4torsion}
Let $\OO$ be a hyperbolic link orbifold such that all torsion has order at least $4$. Then $\vol(\OO) \geq 0.5074...$, with equality if and only if $\OO$ is the figure--$8$ knot in $S^3$, labeled $4$.
\end{theorem}

There are analogues of Theorem \ref{thm:4torsion} for link orbifolds whose torsion orders are bounded below by $n$, for any $n \geq 4$. We will explore these in future work.

\subsection*{Organization}

Here are the main steps in the proofs of Theorems \ref{thm:knots-hom}, \ref{thm:2torsion-nullhom}, and \ref{thm:4torsion}.

The first key step in the proof is to show that for most purposes, it suffices to restrict attention to orbifolds whose only torsion orders are $2$ and/or $3$. This is proved in Section \ref{sec:reduction}; see Theorem \ref{thm:restrictions} and Proposition \ref{prop:reduce-no2} for precise statements. To prove the simple--seeming statement of Theorem \ref{thm:restrictions}, we draw on a number of results and techniques, including the orbifold geometrization theorem  \cite{blp:orbifold, chk:orbifold}, Dunbar's classification of Euclidean orbifolds  \cite{dunbar}, collar estimates due to Gehring, Marshall, and Martin \cite{gmm:axial-distances, gehring-martin:commutators-collars}, and volume estimates for totally geodesic boundary due to Atkinson and Rafalski \cite{atkinson-rafalski}.

In Section \ref{sec:null-hom}, we explore the homological hypotheses that imply Theorem \ref{thm:knots-hom}. The main idea is described in Lemma \ref{lemma:branched-cover}: if the singular locus of $\OO$ has torsion order $n$ and is null-homologous $\mod n$, then $\OO$ has an $n$--fold manifold cover. Thus, once we know the only torsion orders are $2$ and/or $3$, we may take $2$-- and $3$--fold covers to obtain a hyperbolic manifold, whose volume must be at least $\vol(M_W) > 0.94$. After clearing some technical obstacles that arise when both torsion orders are present, this method proves Theorem \ref{thm:knots-hom}.

Section \ref{s:drill-fill} is devoted to proving Theorem \ref{thm:2torsion-nullhom}. The overall line of argument is inspired by the work of Gabai, Meyerhoff, and Milley \cite{moms, milley}: first, drill out the singular locus $\Sigma_\OO$, and then fill it back in. In ``generic'' circumstances, we have precise control over the change in volume under both operations. For drilling, this comes from the work of Agol and Dunfield \cite{ast} and for filling, from the work of Futer, Kalfagianni, and Purcell \cite{fkp:volume}.

To finish the proof of Theorem \ref{thm:2torsion-nullhom}, we must deal with a finite number of ``exceptional,'' non-generic situations. This means drilling a geodesic without a large embedded collar, or Dehn filling along a short slope. The verification of these finitely many (under 200) exceptional cases relies on rigorous computer assistance. For non-generic drilling, we use the program Tube \cite{snap-tube} and the work of Gehring, Machlachlan, Martin, and Reid \cite{GMMR}; see Appendix \ref{s:smalltube}. For short filling, we use Snap and the work of Milley \cite{milley} and Moser \cite{moser:thesis}; see Proposition \ref{prop:moser-search}. The auxiliary files attached to this paper \cite{af-data} provide details (code and output) for these rigorous searches.

Finally, Section \ref{s:4torsion} is devoted to the proof of Theorem \ref{thm:4torsion}. The line of proof is the same as in Theorem  \ref{thm:2torsion-nullhom}, but with fewer technical difficulties. First, we use an analysis similar to that of Section \ref{sec:reduction} to restrict attention to orbifolds whose only torsion order is $4$. Then, we drill and re-fill the singular locus, controlling volume via \cite{ast, fkp:volume}, and check finitely many cases by computer to finish the proof.

\subsection*{Acknowledgments}
We thank Marc Culler, 
Gaven Martin, 
Peter Shalen, and Genevieve Walsh for enlightening and helpful correspondence.
In addition, we are grateful to Nathan Dunfield, Craig Hodgson, Kerry Jones, Peter Milley, and Jessica Purcell for their assistance with the computational aspects of our project.

\section{Reduction to torsion orders $2$ and $3$}\label{sec:reduction}

The goal of this section is to restrict the possibilities for the smallest--volume link orbifold. We will show that the smallest such orbifold cannot have cusps; in other words, the base space is a closed $3$--manifold. In addition, we will show that for most purposes, it suffices to restrict attention to orbifolds whose only torsion orders are $2$ and/or $3$.

\begin{theorem}\label{thm:restrictions}
Let $\Omin$ be a hyperbolic link orbifold, whose volume is minimal among all link orbifolds. Then
\begin{enumerate}
\item\label{i:closed} The base space of $\Omin$ is a closed $3$--manifold.
\item\label{i:only23} The only possible torsion orders of $\Omin$ are $2$ and/or $3$.
\end{enumerate}
\end{theorem}

The proof of Theorem \ref{thm:restrictions} relies on several lemmas. As a quick preview of the argument, we note that $\vol(\Omin) \leq \vol(\OL) = 0.1571...$, where $\OL$ is the link orbifold of Figure \ref{fig:knots}. With this knowledge, Lemma \ref{lemma:no-cusps} implies that the base space of $\Omin$ is closed, and Proposition \ref{prop:reduction} implies the conclusion about torsion orders.

\begin{lemma}\label{lemma:no-cusps}
Let $\OO$ be a hyperbolic link orbifold whose base space is not closed. Then $\vol(\OO) \geq v_3 / 2$, where $v_3 = 1.0149...$ is the volume of a regular ideal tetrahedron. 
%
\end{lemma}

\begin{proof}
Suppose $X_\OO$ is not closed. Then $\bdy X_\OO$ is a surface, which cannot contain any singular points because $\Sigma_\OO$ is a union of closed curves. If some component of $\bdy X_\OO$ has negative Euler characteristic, then $\OO$ has infinite volume. Thus it remains to consider the case where $\bdy X_\OO$ consists of tori, and $\OO$ has one or more torus cusps.

Let $C$ be a maximal cusp of $\OO$: that is, a horospherical cusp neighborhood expanded maximally until it bumps into itself. Using a horoball packing argument, Meyerhoff  showed that $\vol(C) \geq \sqrt{3}/4$. (Meyerhoff's theorem \cite[Section 5]{meyerhoff:first-estimate} is stated for manifolds, but his argument applies verbatim to an orbifold with a torus cusp.)
Furthermore, results of B\"or\"oczky \cite{boroczky} imply that the cusp neighborhood $C$ contains at most $\sqrt{3}/2v_3$ of the total volume of $\OO$. Therefore,
$$\vol(\OO) \: \geq \: \frac{2 v_3}{\sqrt{3}} \, \vol(C) \: \geq \: \frac{2 v_3}{\sqrt{3}} \cdot \frac{\sqrt{3}}{4} \: = \: v_3 / 2.$$

\vspace{-2ex}
\end{proof}

The next lemma takes a step toward proving part \eqref{i:only23} of Theorem \ref{thm:restrictions}, placing a restriction on the torsion orders of a minimal--volume link orbifold.

\begin{lemma}\label{lemma:7estimate}
Let $\OO$ be a hyperbolic link orbifold that contains $p$--torsion for $p \geq 7$. Then $\vol(\OO) \geq 0.1658...$.
\end{lemma}

The proof of this lemma relies on Proposition \ref{drillbound} from Section \ref{s:drill-fill}. Since Proposition \ref{drillbound} does not refer to any results in this paper, we may use it without any risk of circularity.

\begin{proof}
Let $\alpha$ be a component of the singular locus $\Sigma_\OO$ that has torsion order $p \geq 7$. Let $r$ be the \emph{collar radius} of $\alpha$, that is the maximal radius of an embedded tube about $\alpha$. (See Definition \ref{def:simple-collar} for a detailed discussion of collar radius.)

Gehring, Marshall, and Martin have done extensive work on estimating collar radii \cite{gmm:axial-distances, gehring-martin:commutators-collars}. In the setting where $p \geq 7$, they proved that the collar radius $r$ satisfies
$$\cosh(2r) \: \geq \: \frac{\cos( 2\pi / p)}{2 \sin^2 (\pi/p) } \: = \: \frac{\csc^2(\pi/p)}{2}   - 1,$$
which is an increasing function of $p$. Setting $p=7$ and solving for $r$, we obtain
$$ r \: \geq \: \frac{1}{2} \, \cosh^{-1} \left(  \frac{\csc^2(\pi/7)}{2}   - 1 \right) \: = \: 0.54527...$$

Let $\QQ = \OO \setminus \alpha$ be the cusped hyperbolic link orbifold obtained by drilling out $\alpha$. The above lower bound on $r$, combined with Proposition \ref{drillbound} in Section \ref{s:drill-fill}, implies that 
$$ \vol(\OO) \: \geq \: \frac{ \vol(\QQ)} {3.06} \, .$$
Plugging in the estimate $\vol(\QQ) \geq v_3/2 \geq 0.50745$ from Lemma \ref{lemma:no-cusps} completes the proof.
 \end{proof}

To conclude the proof of Theorem \ref{thm:restrictions}, we will need to start with an arbitrary link orbifold $\OO$ and \emph{reduce} to a new orbifold $\OO'$, whose  only torsion orders are $2$ and/or $3$.

\begin{define}\label{def:reduction}
Let $\OO$ be a link orbifold. We construct a new \emph{reduced orbifold} $\OO'$, whose base space and singular locus are (topologically) the same as those of $\OO$. Each component of  the singular locus $\Sigma_\OO$ that is labeled $2$ carries the same label in $\Sigma_{\OO'}$. However, each component of $\Sigma_\OO$ that is labeled  $p \geq 3$ will be relabeled $3$  in $\Sigma_{\OO'}$. \end{define}

\begin{proposition}\label{prop:reduction}
Let $\OO$ be a hyperbolic link orbifold, and let $\OO'$ be the reduced orbifold, constructed as in Definition \ref{def:reduction}. Then one of the following holds true.
\begin{enumerate}
\item\label{i:schlafli} $\OO'$ is hyperbolic. In this case, $\vol(\OO) \geq \vol(\OO')$, with equality if and only if $\OO = \OO'$.

\item\label{i:dunbar} $\OO'$ is geometric but not hyperbolic. In this case,
$\OO'$ is one of the two orbifolds depicted in Figure \ref{fig:dunbar}, and
$\vol(\OO) \geq 0.2537...$.. If $\Sigma_\OO$ is a knot, then $\vol(\OO) \geq
0.5074...$.

\item\label{i:turnover} $\OO'$ is not geometric. In this case, $\OO$ contains
an essential turnover (a sphere that intersects the singular locus $3$
times), and $\vol(\OO) \geq 0.1658...$.

\end{enumerate}
\end{proposition}

\begin{figure}
\labellist
\small\hair 2pt
\pinlabel {$3$} [r] at 12 104
\pinlabel {$2$} [r] at 157 61
\pinlabel {$3$} [br] at 196 102
\endlabellist
\includegraphics[scale=0.8]{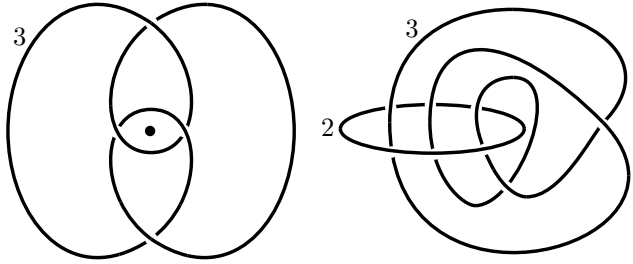}
\caption{The only link orbifolds that admit a Seifert fibered geometry but are not themselves Seifert fibered. Both are Euclidean, and have base space $S^3$. 
Note that the orbifold on the left has a two-fold symmetry, namely $\pi$--rotation about the center point. The quotient is the orbifold on the right.}
\label{fig:dunbar}
\end{figure}

Conclusion \eqref{i:schlafli} is the ``generic'' case of the proposition,
conclusion \eqref{i:turnover} is quite special, and conclusion
\eqref{i:dunbar} is extremely special. In the generic case, we may pass from
studying $\OO$ to studying $\OO'$. In the non-generic cases, we get a lower
bound on $\vol(\OO)$ anyhow.

\begin{proof}
By the orbifold geometrization theorem \cite{blp:orbifold, chk:orbifold},
$\OO'$ is either hyperbolic, or carries a Seifert fibered geometry, or else
has a non-trivial geometric decomposition. We will examine these
possibilities one by one.

First, suppose $\OO'$ is hyperbolic. Then, Kojima showed
\cite{kojima:cone-deformation} that the cone angles on the singular locus
$\Sigma_{\OO'}$ can be continuously deformed downward to those of
$\Sigma_\OO$. By the Schl\"afli formula for cone manifolds \cite[Theorem
3.20]{chk:orbifold}, the volume will strictly increase under this
deformation, unless the deformation is trivial and $\OO' = \OO$. In other
words, $\vol(\OO) \geq \vol(\OO')$, with equality if and only if $\OO =
\OO'$.

The other two cases are handled by lemmas.

\begin{lemma}\label{lemma:dunbar}
Suppose $\OO$ is a hyperbolic link orbifold, and the reduced orbifold $\OO'$ is geometric but not hyperbolic. Then $\OO'$ is one of the two orbifolds depicted in Figure \ref{fig:dunbar}. If $\Sigma_\OO$ is the figure $8$ knot, as in Figure \ref{fig:dunbar} (left), then $\vol(\OO) \geq 0.5074...$. If $\Sigma_\OO$ is the link $6^2_2$,  as in Figure \ref{fig:dunbar} (right), then $\vol(\OO) \geq 0.2537...$.
\end{lemma}

\begin{proof}
The $3$--orbifolds that are geometric but not hyperbolic were classified by Dunbar \cite{dunbar}. His work implies that $\OO'$ must be Seifert fibered, or have solv geometry, or be one of $30$ enumerated exceptions (which have constant--curvature metrics but are not Seifert fibered).

By construction, the singular locus of $\OO'$ is $\Sigma_{\OO'} = \Sigma^2_{\OO'} \cup \Sigma^3_{\OO'}$, where the first of these is labeled $2$ and the second is labeled $3$. 
Similarly, the hyperbolic orbifold $\OO$ has singular locus $\Sigma_{\OO} = \Sigma^2_{\OO} \cup \Sigma^{\geq 3}_{\OO}$. Definition \ref{def:reduction} can be reworded to say that 
$\OO' \setminus \Sigma^{3}_{\OO'} \cong \OO \setminus \Sigma^{\geq 3}_{\OO}.$

Suppose, first, that $\OO'$ is Seifert fibered. By \cite[Proposition 2.41]{chk:orbifold}, $\OO'$ must have a one- or two-sheeted cover $\PP'$, whose singular locus $\Sigma_{\PP'}$ is a union of fibers. The $3$--torsion locus $\Sigma^3_{\OO'}$ must still be singular after being lifted to a double cover, which means that it is the projection of $\Sigma^3_{\PP'}$, hence a union of fibers in $\OO'$. In other words, the cover $\PP' \to \OO'$ restricts to a one- or two-sheeted cover
$$\PP' \setminus \Sigma^{3}_{\PP'} \: \to \: \OO' \setminus \Sigma^{3}_{\OO'} \: \cong \:  \OO \setminus \Sigma^{\geq 3}_{\OO}$$
of Seifert fibered orbifolds. On the other hand, $\OO \setminus \Sigma^{\geq 3}_{\OO}$ must be hyperbolic by Kojima's theorem \cite{kojima:cone-deformation}, which is a contradiction. 

If $\OO'$ has solv geometry, then $\OO'$ must contain an essential torus or essential pillowcase. The cone points on the corners of a  pillowcase are all labeled $2$, i.e.\ the same in $\Sigma_{\OO'}$ as in $\Sigma_\OO$. Therefore, this torus or pillowcase would again contradict the hyperbolicity of $\OO$.
 
We conclude that $\OO'$ must be one of the exceptional orbifolds enumerated by Dunbar. The only link orbifolds on his list are the two Euclidean orbifolds shown in Figure \ref{fig:dunbar}. Compare \cite[Example 2.33]{chk:orbifold}. It remains to estimate the volume of $\OO$. 

If $\Sigma_{\OO'}$ is the figure--$8$ knot labeled $3$, as in Figure \ref{fig:dunbar} (left), then Schl\"afli's formula \cite[Theorem 3.20]{chk:orbifold} implies that the volume of $\OO$ is at least as large as that of the figure--$8$ knot labeled $4$. That orbifold has a four-fold manifold cover $N$, namely the cyclic cover of $S^3$ branched over the figure--$8$ knot. 
Using the work of Moser \cite{moser:thesis}, one may rigorously check that $\vol(N) = 2v_3 = 2.02988...$\protect\footnote{See the proof of Proposition \ref{prop:moser-search}  for a detailed discussion of rigorous volume estimates using Snap and Moser's work. This particular manifold $N$, namely the cyclic four-fold cover of $S^3$ branched over the figure--$8$ knot, is also discussed in the proof of Theorem \ref{thm:4torsion} in Section \ref{s:4torsion}.} Therefore,
$$\vol(\OO) \: \geq \: \vol(N) \, /4 \: = \: v_3 /2  \: = \: 0.50747...$$

If $\OO'$ is the orbifold in Figure \ref{fig:dunbar} (right), then by Definition \ref{def:reduction} the unknot labeled $2$ must also be labeled $2$ in the hyperbolic orbifold $\OO$. We may then take a branched double cover of $S^3$, branched along this unknot, and obtain an orbifold double cover of $\OO$. As in Figure \ref{fig:dunbar}, this double cover is an orbifold $\PP$ whose singular locus is the figure--$8$ knot. Thus, by the previous paragraph,
$$\vol(\OO) \: = \: \vol(\PP) / 2 \: \geq \: v_3  /4 \: = \: 0.25373...,$$
as desired.
\end{proof}

\begin{lemma}\label{lemma:turnover}
Suppose $\OO$ is a hyperbolic link orbifold, and the reduced orbifold $\OO'$ is not geometric. Then $\OO$ contains an essential turnover, and $\vol(\OO) \geq 0.1658...$.

If in addition, $\OO$ does not contain any $2$--torsion, then $\vol(\OO) \geq 0.4408...$.
\end{lemma}

\begin{proof}
Recall that by the orbifold geometrization theorem \cite{blp:orbifold, chk:orbifold},
$\OO'$ is either hyperbolic, or carries a Seifert fibered geometry, or
has a non-trivial geometric decomposition. The hypotheses of this lemma place us in the last of these cases.

 Let $S' $ be a \emph{maximal} surface in the geometric decomposition of $\OO'$. That is, $S' \subset \OO'$ is an embedded  $2$--orbifold, each component of which is essential and has non-negative Euler characteristic. Furthermore, $S'$ is \emph{maximal} in the sense that no component of $S'$ is parallel to any other, and the only essential $2$--orbifolds in $\OO' \setminus S'$ have negative Euler characteristic.

Topologically, each component of $S'$ is either a torus (with no singular points), or a sphere with $0 \leq c \leq 4$ cone points. Note that if we change the cone angles of $\OO'$ back to those of $\OO$, then $S'$ becomes an essential $2$--orbifold $S \subset \OO$, which is no longer an obstruction to hyperbolicity because $\OO$ is hyperbolic. There is only one possibility where $S'$ obstructs hyperbolicity while $S$ does not: each component of $S$ or $S'$ is a \emph{turnover}, that is a $2$--sphere punctured $3$ times by the singular locus.

Adams and Schoenfeld showed that every turnover component of $S$ is totally geodesic in $\OO$  \cite[Theorem 2.1]{adams-schoenfeld}. We may cut $\OO$ along one of these turnovers, call it $T$, and obtain an orbifold with totally geodesic boundary. Let $p \leq q \leq r$ be the orders of torsion on the three cone points of $T$. If $r \geq 7$, then Lemma \ref{lemma:7estimate} implies that $\vol(\OO) \geq 0.1658...$, as desired. 

Thus we may assume that each of $p,q,r$ is at most $6$.
At the same time, hyperbolicity of $\OO$ implies that $T$ must be a hyperbolic turnover, hence 
$$2 \leq p \leq q \leq r \leq 6 \qquad \mbox{and} \qquad  \frac{1}{p}+\frac{1}{q}+\frac{1}{r} \: < \: 1.$$  

There are exactly $24$ integer triples $(p,q,r)$ satisfying the above constraints. For each such triple $(p,q,r)$, we construct and evaluate the volume estimate of Atkinson and Rafalski \cite[Theorem 3.4]{atkinson-rafalski}, which is based on Miyamoto's theorem \cite{miyamoto}. According to \cite[Theorem 3.4]{atkinson-rafalski},
$$\vol(\OO) \geq 0.28248...,$$
with the lowest--volume scenario occurring when $T$ is a $(2,4,5)$ turnover. 

Now, suppose that all torsion orders of $\OO$ are at least $3$. Since the torsion orders of $\OO'$ have been reduced to $3$, every turnover of $S' \subset \OO'$ is a $(3,3,3)$ turnover. We construct a new orbifold $\OO''$, in which every component of $\Sigma_\OO$ that was labeled $p \geq 4$ gets relabeled with torsion order $4$. This will not create any $(3,3,3)$ triples, hence $\OO''$ is hyperbolic. 

By the Schl\"afli formula, $\vol(\OO) \geq \vol(\OO'')$. Furthermore, a turnover $T \subset \OO$ corresponds to a (totally geodesic) turnover $T'' \subset \OO''$, whose cone points have labels $(p,q,r)$ with $3  \leq p \leq q \leq r = 4$. There are three possibilities. Applying the Atkinson--Rafalski estimate \cite[Theorem 3.4]{atkinson-rafalski}  to $\OO''$ gives
$$\vol(\OO) \geq \vol(\OO'') \geq 0.44089...,$$
with the lowest--volume scenario occurring when $T$ is a $(3,3,4)$ turnover.
\end{proof}

Lemmas \ref{lemma:dunbar} and \ref{lemma:turnover} complete the proof of Proposition \ref{prop:reduction}.
\end{proof}

We may also complete the proof of Theorem \ref{thm:restrictions}.

\begin{proof}[Proof of Theorem \ref{thm:restrictions}]
Let $\Omin$ be a link orbifold of minimal volume. By considering the link orbifold $\OL$ in Figure \ref{fig:knots}, we know that $\vol(\Omin) \leq \vol(\OL) < 0.16$. Thus, by Lemma \ref{lemma:no-cusps}, the base space of $\Omin$ must be closed.

Now, consider what Proposition \ref{prop:reduction} says about $\Omin$. Since $\vol(\Omin) < 0.16$, we must be in Case \eqref{i:schlafli} of the proposition: the reduced orbifold $\Omin'$ is hyperbolic. Furthermore, since $\Omin$ has minimal volume by hypothesis, we have $\vol(\Omin) = \vol(\Omin')$. Thus $\Omin = \Omin'$, and this orbifold has only $2$-- and/or $3$--torsion by construction.
\end{proof}

Another immediate consequence of Proposition \ref{prop:reduction} and its proof is the following result, which will be useful in Sections \ref{sec:null-hom} and \ref{s:drill-fill}.

\begin{proposition}\label{prop:reduce-no2}
Let $\OO$ be a hyperbolic link orbifold without any $2$--torsion, and let $\OO'$ be the reduced orbifold, obtained from $\OO$ by changing all cone labels to $3$. Then one of the following holds true.
\begin{enumerate}
\item\label{no2schlafli} $\OO'$ is hyperbolic. In this case, $\vol(\OO) \geq \vol(\OO')$, with equality if and only if $\OO = \OO'$.

\item\label{no2dunbar} $\OO'$ is geometric but not hyperbolic. In this case,  $\Sigma_\OO$ is the figure-8 knot, and $\vol(\OO) \geq 0.5074...$.

\item\label{no2turnover} $\OO'$ is not geometric. In this case, $\OO$ contains an essential turnover, and $\vol(\OO) \geq 0.4408...$.

\end{enumerate}
\end{proposition}

\begin{proof}
Conclusion \eqref{no2schlafli} is identical to \eqref{i:schlafli} from Proposition \ref{prop:reduction}. Conclusion \eqref{no2dunbar} immediately follows from \eqref{i:dunbar} of Proposition \ref{prop:reduction}, because the exceptional link orbifold in Figure \ref{fig:dunbar} (right) has a component labeled $2$. Finally, conclusion \eqref{no2turnover} follows from Lemma \ref{lemma:turnover}.
\end{proof}

\section{Mod $p$ homology and volume estimates}\label{sec:null-hom}

The main goal of this section is to prove Theorem \ref{thm:knots-hom}. The homological hypotheses of this theorem will allow us to find small--degree manifold covers of these orbifolds, via the following lemma.

\begin{lemma}\label{lemma:branched-cover}
Let $X$ be a closed oriented $3$--manifold, $L = L_1 \cup \ldots \cup L_k$ an (oriented) link in $X$, and $n$ a positive integer. Then the following are equivalent:
\begin{enumerate}
\item\label{i:null-hom} There exist coefficients $c_1, \ldots, c_k \in \ZZ_n$, with each $c_i$ a generator of $\ZZ_n$, such that $\sum c_i L_i = 0 \in H_1(X; \ZZ_n)$.
\item\label{i:homomorphism} There is a homomorphism $\varphi: \pi_1(X \setminus L) \to \ZZ_n$,  which sends the meridian of each $L_i$  to a generator.
\item\label{i:cover} There is an $n$--fold cyclic branched cover $Y \to X$, which is branched over $L$.
\end{enumerate} 
\end{lemma}

The statement of Lemma \ref{lemma:branched-cover} is very natural, and likely known to many experts. However, to the best of our knowledge, this was not previously observed in the literature. We are grateful to Marc Culler and Peter Shalen for helpful discussions about this topic.

\begin{proof}
For $\eqref{i:null-hom} \Rightarrow \eqref{i:homomorphism}$, let $\omega$ be a $2$--chain with coefficients in $\ZZ_n$, such that $\bdy \omega = \sum c_i L_i$. Then $\omega$ represents a non-trivial relative homology class in $H_2 (X \setminus N(L), \bdy N(L); \, \ZZ_n)$. The Poincar\'e dual of $\omega$ is a cohomology class in $H^1  (X\setminus N(L) ; \, \ZZ_n)$, i.e.\ a homomorphism $\psi: H_1 (X  \setminus N(L); \,  \ZZ_n) \to \ZZ_n$.  By pre-composing $\psi$ with abelianization and reduction mod $n$, we obtain a homomorphism $\varphi: \pi_1 (X \setminus L) \to \ZZ_n$. By construction, $\varphi$ maps the meridian of $L_i$ to the coefficient $c_i$.

For $\eqref{i:homomorphism} \Rightarrow \eqref{i:null-hom}$, we may reverse the argument above. A homomorphism $\varphi: \pi_1 (X \setminus L) \to \ZZ_n$ must factor through $\psi: H_1 (X  \setminus N(L); \,  \ZZ_n) \to \ZZ_n$. Viewing $\psi$ as a cohomology class, we obtain its Poincar\'e dual $\omega \in H_2 (X \setminus N(L), \bdy N(L); \, \ZZ_n)$. Then $\bdy \omega$ is represented by a $1$--cycle, consisting of curves on each torus of $ \bdy N(L)$. The intersection of $\bdy \omega$ with the meridian $\mu_i$ of $L_i$ is exactly $\varphi(\mu_i) = c_i$, which is a generator of $\ZZ_n$ by hypothesis. Thus, by collapsing the tubular neighborhood $N(L_i)$ to $L_i$ itself, we obtain $\bdy \omega = \sum c_i L_i$, as desired.

For $\eqref{i:homomorphism} \Rightarrow \eqref{i:cover}$, observe that by hypothesis, the homomorphism $\varphi$ is onto. The index--$n$ subgroup $\ker( \varphi)$  corresponds to an $n$--sheeted cyclic cover $Y' \to X \setminus N(L)$. Since the meridian $\mu_i$ of $L_i$ maps to a generator $c_i \in \ZZ_n$, we conclude that $\mu_i^n$ lifts to a simple closed curve in $Y'$, but no smaller power of $\mu_i$ lifts. By Dehn filling $Y'$ along the curves $(\mu_1^n, \ldots, \mu_k^n)$, we obtain a compact manifold $Y$, which is a branched cover of $X$ with branching locus $L$.

For $\eqref{i:cover} \Rightarrow \eqref{i:homomorphism}$, suppose we have a cyclic branched cover $Y \to X$. By removing  the branch locus $L$, we obtain an $n$--sheeted, unbranched cover $Y' \to X \setminus L$, with deck group $\ZZ_n$. The homomorphism $\varphi: \pi_1(X \setminus L) \to \ZZ_n$ records the deck transformation corresponding to each element of $\pi_1(X \setminus L)$. Because the cover is branched $n$ times over each component $L_i$ of $L$, the meridian $\mu_i$ of $L_i$ induces a deck transformation that generates $\ZZ_n$.
\end{proof}

\begin{example}\label{ex:meridian-filling}
Consider the SnapPea census manifolds $\tt m004$, $\tt m006$, $\tt m007$, $\tt m009$, and $\tt m015$. Let $M$ denote one of these five manifolds, and let $X = M(1,0)$ denote its $(1,0)$ Dehn filling. For each $M$, 
we can use Snap and Sage to confirm that the $(1,0)$ slope on the cusp generates a $\ZZ$ summand in $H_1(M, \ZZ) \cong \ZZ \oplus \mbox{(Torsion)}$. Thus, for every positive integer $n$,  there is a chain of surjective homomorphisms
$$ \pi_1(M) \to H_1(M; \ZZ) \to \ZZ \to \ZZ_n,$$
such that the composition  $\varphi: \pi_1(M) \to \ZZ_n$ maps the $(1,0)$ slope to a generator. Note that the $(1,0)$ slope is exactly the meridian of $L$, where $L \subset X$ is the core of the filled solid torus.
Therefore, by Lemma \ref{lemma:branched-cover}, there is an $n$--fold cyclic branched cover of $X$, branched over $L$.

This conclusion has the following interpretation in the language of orbifolds. Let $\OO$ be an orbifold whose base space is $X_\OO = X$, and whose singular locus is $\Sigma_\OO = L$, labeled $n$. (In the language of Section \ref{s:drill-fill}, $\OO$ is the result of $(n,0)$ Dehn filling on the census manifold $M$.) Now, this branched cover of $X_\OO$ branched over $L = \Sigma_\OO$ is an honest manifold cover of $\OO$. In other words, for every $n$, the orbifold $M(n,0)$ has an $n$--fold manifold cover.
\end{example}

The following result is a useful stepping stone toward Theorem \ref{thm:knots-hom}. It is also slightly more general, in the sense that  it only places homological hypotheses on the link $L = \Sigma_\OO$, rather than on the whole base space $X_\OO$.

\begin{theorem}\label{thm:3null}
Let $\OO$ be a hyperbolic link orbifold with no $2$--torsion. Let $\Sigma_\OO = L_1 \cup \ldots \cup L_k$, and suppose that the $L_i$ can be oriented so that $\sum L_i = 0 \in H_1(X_\OO, \ZZ_3)$. Then $\vol(\OO) \geq
\vol(\OK) = 0.31423...$, with equality if and only if $\OO = \OK$.
\end{theorem}

\begin{proof} 
As in Definition \ref{def:reduction}, let $\OO'$ be the reduced orbifold
obtained by changing all cone labels of $\OO$ to $3$. By Proposition \ref{prop:reduce-no2}, either $\vol(\OO) \geq 0.4408...$, which is larger than required, or $\OO'$ is hyperbolic and $\vol(\OO) \geq \vol(\OO')$. We restrict attention to the latter case.

By Lemma \ref{lemma:no-cusps}, the base space $X_\OO = X_{\OO'}$ is closed, or else the volume is again larger than required. Thus we may use the homological tools of this section.

Since we have assumed that  $\sum L_i = 0 \in H_1(X_{\OO'};\ZZ_3)$,
Lemma~\ref{lemma:branched-cover} implies that there is a $3$--fold cyclic
branched cover $M \to X_{\OO'}$, branched over $\Sigma_{\OO'}$.  Since the cone
labels of $\OO'$ are all $3$, the singular geodesics of $\Sigma_{\OO'}$ lift to non-singular ones, and $M$ is a hyperbolic manifold.  By Theorem \ref{thm:weeks}, $\vol(M) \geq \vol(M_W)$, with equality if and only if $M$ is the Weeks manifold $M_W$. Therefore, $\vol(\OO') \geq \vol(M_W) / 3 = 0.31423...$

The symmetry group of $M_W$ is the dihedral group of order $6$,
which has a unique subgroup of order $3$.  The quotient of $M_W$ under this subgroup is precisely the orbifold $\OK$ of Figure \ref{fig:knots}, left. Therefore, the volume of $\OO$ is minimal if and only if $\OO = \OO'  = \OK$.
\end{proof}

We can now prove Theorem \ref{thm:knots-hom} from the introduction.

\begin{named}{Theorem \ref{thm:knots-hom}}
Let $\OO$ be a hyperbolic  link orbifold whose base space has $H_1( X_\OO; \ZZ_6) = 0$.
\begin{enumerate}
\item\label{i:knots} If the singular locus $\Sigma_\OO$ is a knot, then $\vol(\OO) \geq 0.31423...$, with equality if and only if $\OO = \OK$.
\item\label{i:links} If the singular locus $\Sigma_\OO$ is a link with multiple components, then $\vol(\OO) \geq 0.15711...$, with equality if
and only if $\OO = \OL$.
\end{enumerate}
\end{named}

\begin{proof}
First, suppose that the singular locus $\Sigma_\OO$ is a knot. 
Since $X_\OO$ is a $\ZZ_6$--homology sphere, it is also a $\ZZ_2$--homology sphere.
Thus, if $\OO$ has $2$--torsion, Lemma~\ref{lemma:branched-cover} implies that it has a $2$--sheeted manifold cover, which implies that $\vol(\OO) \geq \vol(M_W) / 2 > 0.47$. Here, as above, $M_W$ is the Weeks manifold. 

Otherwise, if $\OO$ has no $2$--torsion, observe that $X_\OO$ is a $\ZZ_3$--homology sphere, and apply Theorem~\ref{thm:3null}. This proves \eqref{i:knots}.

\smallskip

Next, suppose that the singular locus $\Sigma_\OO$ has multiple components.
We may decompose $\Sigma_\OO$ into two disjoint links: $\Sigma_\OO = \Sigma_{\OO}^2 \cup \Sigma_{\OO}^{\geq 3}$, where $\Sigma_{\OO}^2$ is the locus of $2$--torsion, and $\Sigma_{\OO}^{\geq 3}$ is the rest of the singular locus. If $\Sigma_{\OO}^2 = \emptyset$, Theorem \ref{thm:3null} immediately applies. Similarly, if $\Sigma_{\OO}^{\geq 3} = \emptyset$, then $\OO$ has a two-sheeted manifold cover by Lemma \ref{lemma:branched-cover}, hence its volume is at least $\vol(M_W)/2 > 0.47$. Thus we may assume that each of $\Sigma_{\OO}^2$ and $\Sigma_{\OO}^{\geq 3}$ is non-empty.

By Lemma \ref{lemma:branched-cover}, there is a two-fold branched cover $Y \to X_\OO$, branched over $\Sigma_{\OO}^2$. We may pull back the hyperbolic metric on $\OO$ to give a singular hyperbolic metric on $Y$. Since the cover is branched over $\Sigma_{\OO}^2$, each component of $\Sigma_{\OO}^2$ pulls back to a non-singular geodesic. Meanwhile, each component of $\Sigma_{\OO}^{\geq 3}$ is disjoint from the branching locus, hence pulls back to a singular geodesic with the same label as in $\OO$. This gives us a hyperbolic orbifold $\PP$, whose base space is $Y$, and whose singular locus $\Sigma_\PP$ is the preimage of $\Sigma_{\OO}^{\geq 3}$. Note that $\PP$ has no $2$--torsion.

Since the singular locus $\Sigma_{\OO}^{\geq 3}$ is homologically trivial in $H_1(X_\OO, \ZZ_3)$, there is a $2$--chain $\omega$ with $\ZZ_3$ coefficients, such that $\bdy \omega = \Sigma_{\OO}^{\geq 3}$. Without loss of generality, this $2$--chain is smooth, intersects  $\Sigma_{\OO}^2$ transversely, and has a vertex at each point of intersection with $\Sigma_{\OO}^2$. Then $\omega$ pulls back to a $2$--chain $\eta$ in $Y$, also with $\ZZ_3$ coefficients, such that $\bdy \eta = \Sigma_\PP$, i.e.\ the preimage of $\Sigma_{\OO}^{\geq 3}$ in $\PP$. Now, we may apply Theorem \ref{thm:3null} to $\PP$ and conclude
$$\vol(\OO) \: = \: \vol(\PP) / 2 \: \geq \: \vol(M_W) / 6 \: = \: 0.15711...,$$ with equality if
and only if $\PP = \OK$. Finally, it's easy to check (by examining the symmetry group) that $\OL$ is the only two-fold quotient of $\OK$ that is a link orbifold.
\end{proof}

\section{Drilling and filling}\label{s:drill-fill}

The main goal of this section is to prove Theorem~\ref{thm:2torsion-nullhom}, which was stated in the introduction. Most of the effort will go toward proving the second statement of that theorem: 

\begin{theorem}\label{thm:no2torsion}
Let $\OO$ be a hyperbolic link orbifold without any $2$--torsion.  Then $\vol(\OO) \geq
0.2371.$ 
\end{theorem}

\begin{proof}[Proof of Theorem \ref{thm:2torsion-nullhom}, assuming Theorem \ref{thm:no2torsion}]
Let $\OO$ be a link orbifold whose $2$--torsion locus is null-homologous. By Lemma \ref{lemma:no-cusps}, the base space $X_\OO$ is closed, or else $\vol(\OO)$ is already larger than required.
Then Lemma~\ref{lemma:branched-cover} implies there is a double cover $\PP \to \OO$, in which the $2$--torsion locus of $\OO$ is covered by non-singular geodesics. Thus $\PP$ has no $2$--torsion at all, hence $\vol(\OO) \: \geq \: \vol(\PP) / 2 \: \geq \: 0.2371/2.$
\end{proof}

After ruling out some special cases, the proof of Theorem \ref{thm:no2torsion} proceeds in four steps:
\begin{enumerate}[$1.$]
\item\label{step:drill} Drill out the singular locus $\Sigma_\OO$, producing a cusped manifold $\MO$. Estimate the change in volume under this drilling operation, using the work of Agol and Dunfield \cite{ast}.
\item\label{step:mom} Apply the work of Gabai, Meyerhoff, and Milley \cite{moms} to estimate the volume of $\MO$.
\item\label{step:fill} Fill in the singular locus $\Sigma_\OO$ to recover $\OO$. If the Dehn filling curves are long, the change in volume can be bounded using results of Futer, Kalfagianni, and Purcell \cite{fkp:volume}.
\item\label{step:moser} If the Dehn filling curves are short, there are finitely many possibilities. These finitely many cases can be checked by a rigorous computer search, as in Milley \cite{milley}.
\end{enumerate}

This proof strategy is quite similar to the one employed by Gabai, Meyerhoff, and Milley to prove the minimality of the Weeks manifold \cite{moms, milley}.
Just as in their proof, the first step in our argument relies on having lower bounds on the \emph{collar radius} about the singular locus $\Sigma_\OO$. Before proceeding further, we define what this means.

\begin{define}\label{def:simple-collar}
Let $\OO = \HH^3 / \Gamma$ be a hyperbolic $3$--orbifold. A set $S \subset \HH^3$ is called \emph{precisely invariant} under $\Gamma$ if, for all $g \in \Gamma$, $g(S)$ either coincides with $S$ or is disjoint from $S$. A geodesic $\alpha$ in the singular locus of $\OO$ is called \emph{simple} if an arbitrary lift of $\alpha$ to $\HH^3$ is precisely invariant --- in other words, if the lifts of $\alpha$ are disjoint geodesics in $\HH^3$. Note that if $\OO$ is a link orbifold, every component of $\Sigma_\OO$ is simple.

Let $\alpha \subset \Sigma_\OO$ be a simple geodesic, and let $\widetilde{\alpha} \subset \HH^3$ be one  lift of $\alpha$. The \emph{collar radius} of $\alpha$ is the supremum of all $r$ with the property that an $r$--neighborhood of $\widetilde{\alpha}$ is precisely invariant. 
\end{define}

The collar radius of a disjoint union of geodesics is defined in the same
way. In particular, for a link orbifold $\OO$, the collar radius of
$\Sigma_\OO$ is at least $r$ exactly when the $r$--neighborhood of $\Sigma_\OO$ is a disjoint union of tubes.

 Step \ref{step:drill} of our proof strategy is the following result of Agol and Dunfield \cite[Theorem 10.1]{ast}, which was conveniently reformulated by 
Agol, Culler, and Shalen \cite[Lemma 3.1]{agol-culler-shalen}.  We will need to 
generalize the result to orbifolds, using
a straightforward application of Selberg's Lemma.

\begin{proposition}\label{drillbound}
Let $\OO$ be a finite--volume, orientable, hyperbolic $3$--orbifold, and let $L$ be a geodesic link in $\OO$.  Suppose that the collar radius of $L$ is bounded below by $r > 0$. Let $\QQ$ be the cusped hyperbolic orbifold obtained by drilling out $L$. Then
$$\vol(\QQ) \: \leq \:  (\coth^3{2r}) \left( 1 + \frac{0.91}{\cosh{2r}} \right) \vol(\OO) $$
\end{proposition}

\begin{proof}
By Selberg's lemma \cite{alperin:selberg-lemma}, $\OO$ is finitely covered by a hyperbolic manifold $M$. Let $n$ be the degree of the cover.
 Let $L$ be a closed geodesic link in $\OO$, and $\widetilde{L}$
the preimage of $L$ in $M$.  Then $N = M \setminus \widetilde{L}$ is an 
$n$--sheeted cover of $\QQ$.  Let $T \subset M$ be a maximum--radius embedded tube about $\widetilde{L}$. Since an embedded tube in $\OO$ pulls back to an embedded tube in $M$, the radius of $T$ is $r_M \geq r$. With this notation, Agol and Dunfield \cite[Theorem 10.1]{ast} showed that
$$\vol(N) \:  \leq \:  \coth^3(2r_M) \left(\vol(M) + \frac{\vol(T)}{\cosh(2r_M)} \right).$$
(See \cite[page 2299]{agol-culler-shalen} for this formulation of the Agol--Dunfield estimate.)

Now, we apply an estimate of Przeworski \cite[Corollary 4.4]{przeworski}, who showed that
$$\vol(T) \leq 0.91 \, \vol(M).$$
Substituting this into the previous equation, we obtain
\setlength{\jot}{1.3ex}
\begin{eqnarray*}
\vol(N) 
& \leq & \coth^3(2r_M) \left(1 + \frac{0.91}{\cosh(2r_M)}\right)  \vol(M) \\
& \leq & \coth^3(2r) \left(1 + \frac{0.91}{\cosh 2r}\right)  \vol(M),
\end{eqnarray*}
where the last inequality follows because the function  $f(x) = \coth^3(2x) \left(1 + \frac{0.91}{\cosh(2x)}\right)$ is
increasing in $x$. Finally, since volume is   multiplicative under taking finite covers, we have
$\vol(N) = n \vol(\QQ)$ and $\vol(M) = n \vol(\OO)$, completing the proof.
\end{proof}

Step \ref{step:mom} in the proof of Theorem \ref{thm:no2torsion} relies on a result of
Gabai, Meyerhoff, and Milley \cite[Corollary 1.2]{moms}, which we restate in a slightly more general form.

\begin{theorem}\label{thm:moms}
Let $N$ be a cusped orientable hyperbolic $3$--manifold with volume no
more than $2.848.$  Then $N$ has exactly one cusp, and is one of the Snappea census manifolds $\tt m003$,
$\tt m004$, $\tt m006$, $\tt m007$, $\tt m009$, $\tt m010$, $\tt m011$, $\tt m015$, $\tt m016$, or $\tt m017$.
\end{theorem}

We note that although \cite[Corollary 1.2]{moms} is only stated for $1$--cusped manifolds, in fact the result \emph{implies} that $N$ has one cusp. For,  any  multi-cusped manifold of volume $ \leq 2.848$ would have an infinite sequence of $1$--cusped Dehn fillings also satisfying this volume bound, which would contradict their enumeration. In fact, Agol showed 
\cite{agol-2cusped} the smallest volume multi-cusped manifold has $\vol(N) = 3.6638...$.

\smallskip

Step \ref{step:fill} in the proof of Theorem \ref{thm:no2torsion} is to estimate the change in volume under Dehn filling. 

We recall some standard notation about orbifold Dehn filling. If $\QQ$ is a hyperbolic orbifold with a torus cusp $C$, a \emph{slope} $s$ is an unoriented, non-trivial homology class on $\bdy C$. That is, if $H_1(\bdy C) \cong \ZZ^2$ is endowed with a basis $\langle \mu, \lambda \rangle$, then $s$ can be written as $\pm (p \mu + q \lambda)$, where $p$ and $q$ need not be relatively prime. If $p$ and $q$ happen to be coprime, then the Dehn filling $\QQ(s)$ is the result of attaching a non-singular solid torus to $\QQ$, whose meridian disk is mapped to $s$. If $p$ and $q$ are \emph{not} coprime, then $\QQ(s)$ is the result of attaching a singular solid torus to $\QQ$, with the core curve carrying the torsion label $\gcd(p,q)$.

If $s_1, \ldots, s_k$ are slopes on multiple cusps of $\QQ$, the Dehn filling  $\QQ(s_1, ..., s_k)$ is defined in exactly the same way. With this notation, we can state the following mild generalization of a theorem of
Futer, Kalfagianni, and Purcell \cite[Theorem 1.1]{fkp:volume}.

\begin{theorem}\label{thm:fkp}
Let $\QQ$ be a complete, finite--volume hyperbolic $3$--orbifold with torus
cusps.  Suppose that $C_1, ..., C_k$ are disjoint horoball neighborhoods of
some subset of the cusps.  Let $s_1, ..., s_k$ be slopes on $\partial C_1,
..., \partial C_k,$ each with length greater than $2\pi$.  Denote the
minimal slope length by $\lmin$.  Then $\QQ(s_1, ..., s_k)$ is
hyperbolic, and
$$\vol(\QQ(s_1, ..., s_k)) \geq \left( 1- \left(
\frac{2\pi}{\lmin}\right)^2 \right)^{3/2} \vol(\QQ).$$
\end{theorem}

\begin{proof}
By Selberg's lemma, there exist finite--sheeted manifold covers of
$\QQ(s_1, ... , s_k)$.  Let $\pi: M \to \QQ(s_1, ... , s_k)$ be one such covering map,
of degree $n$.  
Then $N = \pi^{-1}(\QQ)$ is an $n$--sheeted manifold cover of $\QQ$.
The hyperbolic metric on $\QQ$ pulls back to a non-singular, complete,
hyperbolic metric on $N$.

Every slope $s_i$ on a cusp torus $\bdy C_i$ can be realized by a (not necessarily simple) Euclidean geodesic. Abusing notation slightly, we denote this geodesic by $s_i$. We claim that, for every cusp $C_i^j$ covering $C_i$, the geodesic $s_i$ lifts to a \emph{simple} Euclidean geodesic of the same length.

Consider a solid torus of $V \subset M \setminus N$. Then the image $\pi(V)$ is a (possibly singular) solid torus of $\QQ(s_1, ... , s_k) \setminus \QQ$. By construction, the meridian disk of $V$ covers the meridian disk of $\pi(V)$, with the meridian circle of $V$ mapping to one of the $s_i$. If the solid torus $\pi(V)$ is non-singular,  $s_i$ will be a primitive slope, and the restriction of $\pi$ to a meridian disk is $1$--$1$. Thus the lift of $s_i$ to $\bdy V$ is a simple closed Euclidean geodesic of the same length.

If $\pi(V)$ is a singular solid torus, its core will have a cone singularity of order $n_i > 1$. Then the restriction of $\pi$ to a meridian disk is an $n_i$--sheeted cover of a singular disk, which restricts to an $n_i$--sheeted cover on the boundary circle. Meanwhile, having a singularity of order $n_i$ means that $s_i = n_i \overline{s_i}$, where $\overline{s_i}$ is a primitive slope on $C_i$. Thus the lift of $s_i$ to $\bdy V$ is a meridian of $V$, which is a simple Euclidean geodesic of length $n_i \ell(\overline{s_i}) = \ell(s_i)$.

We conclude that the set of slope lengths in
$N$ is the same as the set slope lengths in $\QQ$.  Thus, by hypothesis, the shortest slope length is $\lmin > 2\pi$. This implies that $M$ is hyperbolic, hence its quotient $\QQ(s_1, ... , s_k)$ is hyperbolic as well. The volume estimate of the theorem follows  follows from that of
Futer--Kalfagianni--Purcell \cite[Theorem 1.1]{fkp:volume} by noting that $\vol(N) = n \vol(\QQ)$ and that
$\vol(M) = n \vol(\QQ(s_1,..., s_k))$.
\end{proof}

Finally, the last step in the proof of Theorem \ref{thm:no2torsion} is the following technical result.

\begin{proposition}\label{prop:moser-search}
Let $M$ be one of the ten one-cusped manifolds enumerated in Theorem \ref{thm:moms}. Let $\OO$ be a hyperbolic link orbifold with $3$--torsion only, created by filling a cusp of $M$. Then $\vol(\OO) \geq 0.31423...$, with
equality if and only if $\OO = \OK$, the orbifold of Figure \ref{fig:knots}, left.
\end{proposition}

\begin{proof}
Suppose $M$ is one of the cusped manifolds in Theorem \ref{thm:moms}, and $B < \vol(M)$. By Theorem \ref{thm:fkp}, if an orbifold filling $\OO = M(s)$ has volume at most $B$, then the filling slope $s = p\mu + q\lambda$ must be drawn from the set
$$ \FF_M(B) := \left\{ (p,q) \in \ZZ^2 \:  : \: 1- \left(
\frac{2\pi}{\ell(p\mu + q \lambda)}\right)^2 \leq \left( \frac{B}{\vol(M)} \right)^{2/3} \right\} .$$
This is a finite set of integer pairs. For our purposes, we set $B = 0.32$ and consider the fillings that yield orbifolds with $3$--torsion. That is, for each $M$, we consider the set 
$$\FF_M^3(B) = \{(p,q) \in \ZZ^2 \mid (p,q) \in \FF_M(B) \text{ and }
\gcd(p,q) = 3\}.$$
There are a total of $144$ triples $(M,p,q)$ where $M$ is one of the census manifolds of Theorem \ref{thm:moms} and $(p,q) \in \FF_M^3(0.32)$. See the accompanying data set \cite{af-data} for a description of these slopes. Note that we chose the rounded value $B=0.32$ to account for possible round-off error in computing slope lengths.

For each $M$, and each $(p,q) \in
\FF_M^3(B)$, we must check that $\vol(M(s)) \geq \vol(\OK)$.  This can be accomplished by a rigorous computer
search, using code developed by Milley \cite{milley} and Moser \cite{moser:thesis}.  We outline the method briefly, referring to Milley  \cite{milley} for more details. 

For each slope $s=p \mu + q \lambda$, we use Snap \cite{snap-tube} to
attempt to find an ideal hyperbolic triangulation of the orbifold filling $M(s)$. This means that, starting with the canonical triangulation of $M$, Snap attempts to find a solution to the gluing equations in which the holonomy of $s$ is a $2\pi$ rotation. This search can lead to one of three results:
\begin{itemize}
\item Snap may fail to find a solution to the gluing equations.
\item Snap may find an (approximate) solution to the gluing equations, in which all tetrahedra are positively oriented.
\item Snap may find an (approximate) solution to the gluing equations, in which one or more tetrahedra are negatively oriented.
\end{itemize}
We consider each of these in turn.

The only filling for which Snap fails to find a solution to the gluing equations is the $(3,0)$ filling of $\tt m004$. This is the orbifold with base space $S^3$ and singular locus the figure--8 knot labeled $3$, as depicted in Figure \ref{fig:dunbar}, left. As detailed in  \cite[Example 2.33]{chk:orbifold}, this orbifold is Euclidean, and in particular non-hyperbolic.

There are $139$ fillings for which Snap finds a positively oriented
approximate solution to the gluing equations. In each case, we use Milley's
implementation of Moser's algorithm to rigorously verify that there is an
actual positively oriented solution near the solution reported by Snap
\cite{milley, moser:thesis}.  Moser's algorithm also gives an upper bound on the distance between an approximate solution and the true solution. Then, Milley's code uses this upper bound to rigorously verify that the volume
of each filled orbifold is at least $B = 0.32$.

Finally, there are $4$ fillings for which Snap finds a solution with negatively oriented tetrahedra. These are 
$\tt m006(3,0)$, $\tt m007(3,0)$, $\tt m009(3,0)$, and $\tt m015(3,0)$. Each of these orbifolds $\OO$ is discussed in Example \ref{ex:meridian-filling}. In particular, by Example \ref{ex:meridian-filling} and Lemma \ref{lemma:branched-cover}, there is a threefold cyclic cover of $X_\OO$, branched over $\Sigma_\OO$. In other words, the orbifold $\OO$ is triply covered by a closed hyperbolic manifold. Therefore, as in Theorem 
\ref{thm:3null}, we have $\vol(\OO) \geq \vol(M_W)/3$, with
equality if and only if $\OO = \OK$. We remark that $\tt m015(3,0)$ is the minimal orbifold $\OK$.
\end{proof}

We can now combine these ingredients to prove Theorem \ref{thm:no2torsion}.

\begin{proof}[Proof of Theorem~\ref{thm:no2torsion}]
Let $\OO$ be a hyperbolic link orbifold without any $2$--torsion. As in Definition \ref{def:reduction}, let $\OO'$ be the reduced orbifold obtained from $\OO$ by changing all torsion orders to $3$. Proposition  \ref{prop:reduce-no2} says that either $\OO'$ is hyperbolic, in which case $\vol(\OO) \geq \vol(\OO')$, or else $\vol(\OO)$ is much larger than $0.2371$. Thus, for the purpose of proving $\vol(\OO) \geq 0.2371$, it suffices to assume that $\OO$ has $3$--torsion only.

Let $r$ be the collar radius of $\Sigma_\OO$. If $r \leq 0.294$, then Theorem \ref{smalltube} in the Appendix says that $\vol(\OO) \geq 0.3142$. Thus we may assume that $r > 0.294$.

Let $\MO = X_\OO \setminus
\Sigma_\OO$ be the cusped manifold obtained by drilling out the singular locus of
$\OO$.  Since $r> 0.294$,
Proposition~\ref{drillbound} says 
$$\frac{\vol(\MO)}{\vol(\OO)} \: \leq \:
\coth^3(2\cdot0.294) \left( 1 + \frac{0.91}{\cosh(2\cdot 0.294)} \right) \:< \:
12.011.$$
 Then,  if $\vol(\MO) \geq 2.848$, we conclude that $\vol(\OO) > 0.2371.$

If $\vol(\MO) < 2.848,$ then Theorem~\ref{thm:moms} implies that
$\MO$ is one of ten $1$--cusped manifolds.  
Now, Proposition \ref{prop:moser-search} gives $\vol(\OO) > 0.3142$, completing the proof. 
\end{proof}

\section{Orbifolds with high torsion}\label{s:4torsion}

The goal of this section is to prove Theorem \ref{thm:4torsion}, which identifies the unique smallest volume link orbifold without any $2$-- or $3$--torsion. The method of proof is similar to that of Theorem \ref{thm:no2torsion} in Section \ref{s:drill-fill}, but with a lot fewer special cases. As a result of the streamlined argument, we can identify the unique minimizer.

We begin with the following analogue of Proposition \ref{prop:reduction}.

\begin{lemma}\label{lemma:reduce-4}
Let $\OO$ be a hyperbolic link orbifold such that all torsion has order at least $n \geq 4$, and let $\OO'$ be the link orbifold obtained by changing all torsion labels of $\Sigma_\OO$ to $n$. Then $\OO'$ is hyperbolic. Furthermore, $\vol(\OO) \geq \vol(\OO')$, with equality if and only if $\OO = \OO'$.
\end{lemma}

\begin{proof}
The proof closely parallels the proof of Proposition \ref{prop:reduction}, but with fewer special cases. Recall that by the orbifold geometrization theorem \cite{blp:orbifold, chk:orbifold}, $\OO'$ is either hyperbolic, or carries a Seifert fibered geometry, or else has a non-trivial geometric decomposition. We want to rule out the non-hyperbolic possibilities.

If $\OO'$ is geometric but not hyperbolic, then we argue as in Lemma \ref{lemma:dunbar}. By that argument, $\OO'$ must be one of the exceptional orbifolds enumerated by Dunbar \cite{dunbar}. But all of the orbifolds in his list have $2$-- and/or $3$--torsion, which means our $\OO'$ cannot be on his list.

If $\OO'$ is not geometric, then we argue as in Lemma \ref{lemma:turnover} to conclude that every surface in the geometric decomposition of $\OO'$ is a turnover. But a turnover with cone points labeled $n \geq 4$ has negative Euler characteristic, i.e.\ would not occur in the geometric decomposition of $\OO'$.

Since all non-hyperbolic cases lead to contradictions, $\OO'$ must be hyperbolic. By Schl\"afli's formula  \cite[Theorem
3.20]{chk:orbifold}, $\vol(\OO) \geq \vol(\OO')$, with equality if and only if $\OO = \OO'$.
\end{proof}

\begin{named}{Theorem \ref{thm:4torsion}}
Let $\OO$ be a hyperbolic link orbifold such that all torsion has order at least $4$. Then $\vol(\OO) \geq 0.5074...$, with equality if and only if $\OO$ is the figure--$8$ knot in $S^3$, labeled $4$.
\end{named}

\begin{proof}
Let $\Sigma = \Sigma_\OO$ be the singular locus of $\OO$. By Lemma \ref{lemma:reduce-4}, we may assume without loss of generality that all cone labels of $\Sigma_\OO$ are exactly $4$.  By a theorem of Gehring and Martin \cite[Theorem 4.20]{gehring-martin:commutators-collars}, the collar radius of $\Sigma_\OO$ is
$$ r \: \geq \: \frac{1}{2} \: \mathrm{arccosh} \frac{1 + \sqrt{3} }{2} \: = \: 0.4157... $$
Let $\MO = X_\OO \setminus \Sigma_\OO$ be the cusped manifold obtained by drilling out the singular locus of
$\OO$. 
By Proposition \ref{drillbound},
$$\frac{\vol(\MO)}{\vol(\OO)} \: \leq \:
\coth^3(2\cdot 0.4157) \left( 1 + \frac{0.91}{\cosh(2\cdot 0.4157)} \right) \:< \:
5.271.$$

Now, consider the volume of the cusped manifold $\MO$. If $\vol(\MO) \geq 2.7$, then
$$\vol(\OO) \: > \: 2.7 / 5.271  \: = \: 0.5122... \, ,$$ which is larger
than the desired bound. Otherwise, if $\vol(\MO) \leq 2.7$, then Theorem
\ref{thm:moms} implies that $\MO$ is one of the $6$ census manifolds $\tt m003$,
$\tt m004$, $\tt m006$, $\tt m007$, $\tt m009$ or $\tt m010$. (The other four manifolds
enumerated in that theorem have volume larger than $2.78$.) To complete the
proof, we perform a computer analysis of short Dehn fillings of these six
manifolds, as in Proposition \ref{prop:moser-search}.

Let $M = \MO$ be one of the six  census manifolds $\tt m003$, $\tt m004$, $\tt m006$,
$\tt m007$, $\tt m009$ or $\tt m010$, and let $B = 0.51$. Importing the notation of
Proposition \ref{prop:moser-search}, consider the set $$ \FF_M^4 (B) :=
\left\{ (p,q) \in \ZZ^2 \:  : \: 1- \left(
\frac{2\pi}{\ell(p\mu + q \lambda)}\right)^2 \leq \left( \frac{B}{\vol(M)}
\right)^{2/3}  \text{ and } \gcd(p,q) = 4\right\} .$$ If the link orbifold
$\OO = M(s)$ is obtained by Dehn filling $M$,  has $4$--torsion only, and
satisfies $\vol(\OO) < B = 0.51$, then Theorem \ref{thm:fkp} implies the
filling slope $s = p \mu + q \lambda$ is one of the finite set of slopes in
$\FF_M^4(B)$. In fact, there are exactly $57$ triples $(M, p, q)$ such that
$M$ is one of $\tt m003$ through $\tt m010$ and $(p,q) \in \FF_M^4(B)$. We analyze
these fillings using Snap.

There are $54$ fillings for which Snap finds a positively oriented
approximate solution to the gluing equations. In each case, Moser's algorithm
\cite{moser:thesis} gives an upper bound on the distance between this
approximate solution and a true solution that gives the hyperbolic structure.
Then, Milley's code \cite{milley} uses this upper bound to rigorously verify
that the volume of each filled orbifold is at least $B = 0.51$.

Finally, there are $3$ fillings for which Snap finds a solution with
negatively oriented tetrahedra: namely,  $\tt m004(4,0)$, $\tt m006(4,0)$, and
$\tt m009(4,0)$. For each of these, we must find an alternate way to estimate its
volume. Note that each of these orbifolds is described in Example \ref{ex:meridian-filling}, and in particular, each has a four-fold, cyclic manifold cover.

The orbifold $\OO = {\tt m004(4,0)}$ is our claimed minimizer: the figure--$8$ knot
labeled $4$. The four-fold manifold cover $N$ is the
cyclic cover of $S^3$ branched over the figure--$8$ knot. After drilling a short curve in $N$ and filling it back in, Snap can get a a triangulation with all positively oriented tetrahedra. (See the auxiliary files \cite{af-data} for a precise description of this triangulation.)
Then,
the Milley--Moser algorithm can certify that $\vol(N) = 2 v_3 = 2.02988...$, hence
$\vol(\OO) = v_3/2 = 0.5074...$.

By Example \ref{ex:meridian-filling}, each of $\tt m006(4,0)$ and $\tt m009(4,0)$ also has a four-fold, cyclic manifold cover. For each of these manifolds, Snap can find a positively oriented triangulation by drilling and refilling a short curve. The precise description of these positively oriented triangulations is given in the accompanying data set \cite{af-data}.
\end{proof}

\appendix

\section{Volume bounds for small collar radius}\label{s:smalltube}

The line of argument in Section \ref{s:drill-fill} relies on having an embedded collar of some size about the singular locus of $\OO$. In this appendix, we classify the link orbifolds whose torsion order is $3$ and whose collar radius is small, in particular bounded above by $0.294$. As it happens, there are only two link orbifolds with this property.

\begin{theorem}\label{smalltube}
Let $\OO$ be a hyperbolic link orbifold with $3$--torsion only.  Suppose that the
collar radius of $\Sigma_\OO$ is at most $0.294$.  Then $\OO$ is one of the following two orbifolds, both with base space $X_\OO = S^3$:
\begin{enumerate}
\item $\OO = \OK$, the knot orbifold of Figure \ref{fig:knots}, left, and $\vol(\OO) = 0.31423...$.

\item $\OO = \OO_W$, the link orbifold of Figure \ref{fig:whitehead}, and $\vol(\OO) = 0.52772...$.

\end{enumerate}
\end{theorem}

\begin{figure}[h]
\labellist
\small\hair 2pt
\pinlabel {$3$} [r] at 1 60
\pinlabel {$3$} [br]  at 25 107
\endlabellist
\includegraphics[scale=0.8]{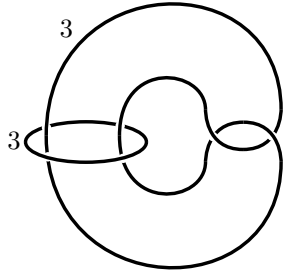}
\caption{The orbifold $\OO_W$, with singular locus the Whitehead link.}
\label{fig:whitehead}
\end{figure}

The hypothesis $r \leq 0.294$, as well as the highly restricted conclusion, comes from the following theorem of Gehring, Machlachlan,
Martin, and Reid \cite{GMMR}.

\begin{theorem}[Gehring--Machlachlan--Martin--Reid]\label{thm:gmmr}
Let $G = \langle f,g \rangle$  be a Kleinian group, generated by an elliptic $f$ of order $3$ and an elliptic $g$ of order $2$. Suppose that the axis of $f$ has collar radius $r \leq 0.294$. Then $G$ is conjugate in $PSL(2, \CC)$ to one of twelve arithmetic groups $G_{3,i}$ for $1 \leq i \leq 12$, enumerated in \cite[Table 6]{GMMR}. See also Table \ref{table:gmmr} below.

Furthermore, if the axis of $f$ is simple, then $G$ is conjugate to one of $G_{3,i}$ for $i = 6,7, 10$.
\end{theorem}

\begin{table}[h]
\begin{tabular}{| c | c | c |}
\hline
$i$ & Collar radius & $2 \times $Co-volume \\
\hline
6 & 0.24486... & 0.31423... \\
7 & 0.24809... & 0.31423... \\
10 & 0.27702... & 0.52772... \\
\hline

\end{tabular}

\vspace{1ex}

\caption{Arithmetic Kleinian groups $G_{3,i}$, generated by a simple elliptic $f_i$ of order $3$ and an elliptic $g_i$ of order $2$.}
\label{table:gmmr}
\end{table}

To apply Theorem \ref{thm:gmmr}, we need to relate our link orbifold to the two-generator groups as above.

\begin{lemma}\label{gavens-lemma}
Let $\Gamma = \pi_1(\OO)$ be a Kleinian group in which all elliptic axes are simple, and have the same torsion order $n \geq 3$. Let $r$ be the collar radius of $\Sigma_\OO$. Then there is a Kleinian group $G = \langle f,g \rangle$, generated by an elliptic $f$ of order $n$ and an elliptic $g$ of order $2$, such that the axis of $f$ is again simple, and has collar radius equal to $r$. Furthermore, an index--$2$ subgroup $H \subset G$ is also a subgroup of $\Gamma$.\end{lemma}

Variants of Lemma \ref{gavens-lemma} appear in the work of Gehring, Marshall and Martin; compare \cite[Lemma 2.26]{gehring-martin:commutators-collars}. We thank Gaven Martin for explaining the proof.

\begin{proof}
Let $r$ be the collar radius of $\Gamma$. This means  that in $\HH^3$, the closest distance between a pair of elliptic axes for $\Gamma$ is precisely $2r$. Let $\alpha, \gamma$ be two such axes at distance $2r$, and let $f, h \in \Gamma$ be the order--$n$ rotations about these axes. Let $H = \langle f, h \rangle$ be the subgroup of $\Gamma$ generated by $f$ and $h$.

Note that $\alpha$ and $\gamma$ do not share any endpoints at infinity; if they did, the distance between them would approach $0$. Hence there is a unique geodesic segment $s$ that meets $\alpha$ and $\gamma$ perpendicularly. The midpoint of $s$ meets a geodesic $\beta$, with the property that $\pi$--rotation about $\beta$ maps $\alpha$ to $\gamma$. The complex length between $\alpha$ and $\beta$ is exactly half that from $\alpha$ to $\gamma$; in particular, $\beta$ lies at distance $r$ from both $\alpha$ and $\gamma$.

Let $g$ be the $\pi$--rotation about $\beta$. Then $gfg = h^{\pm 1}$, hence $H = \langle f, h \rangle =  \langle f, \, gfg \rangle $. Define $G = \langle f,g \rangle$ to be the subgroup of $PSL(2,\CC)$ generated by $f$ and $g$. Then $H$ is precisely the subgroup of $G$ consisting of all elements expressible by a word in $f,g$ with an even number of $g$'s. Hence $G = H \cup gH$, and $H$ has index $2$. Thus, since $H$ is discrete, $G$ is also discrete.

 Finally, observe that the orbit of $\alpha$ under $G = \langle f,g \rangle$ is exactly the same as the orbit of $\alpha \cup \gamma$ under $H =   \langle f, \, gfg \rangle $. Thus a collar of radius $r$ about $\alpha$ is precisely invariant under $G$, and $r$ is the largest value with this property. Since the collar is precisely invariant, then $\alpha$ itself is precisely invariant, hence simple.
\end{proof}

\begin{example}\label{ex:52-knot}
Consider the knot orbifold $\OO_K$ of Figure \ref{fig:knots}. Using the program Tube \cite{snap-tube}, we verify that the singular locus of $\OO_K$ has collar radius $0.24486...$. This means that, in the lift to $\HH^3$, the closest pair of elliptic axes are at distance $2 \times 0.24486...$ from each other. Setting these two axes to be $\alpha$ and $\gamma$, and constructing $\beta$ between them as in the proof of Lemma  \ref{gavens-lemma}, will produce a $2$--generator group $G = \langle f, g \rangle$, where the axis of $f$ has collar radius $0.24486$. By Theorem \ref{thm:gmmr}, this two-generator group must be conjugate to $G_{3,6}$ from Table \ref{table:gmmr}. After conjugation, we assume that $G = G_{3,6}$.

In the universal cover of $\OO_K$, there are also a pair of order--$3$ axes at distance $2 \times 0.24809$ (again, this is computed using Tube). Setting \emph{these} two axes to be $\alpha$ and $\gamma$,  as in the proof of Lemma  \ref{gavens-lemma}, will produce the two-generator group $G_{3,7}$ from Table \ref{table:gmmr}. The identification of $G = \langle f, g \rangle$ with $G_{3,7}$ is rigorous because of Theorem \ref{thm:gmmr}.

Thus, for both $i =6$ and $i=7$, the index--$2$ subgroup $H_{3,i} = \langle f_i , h_i \rangle \subset G_{3,i}$ coincides with a subgroup of $\pi_1(\OO_K)$. Furthermore, since the volume of $\OO_K$ equals the co-volume of $H_{3,i}$, we conclude that $H_{3,i} = \pi_1(\OO_K)$. 
\end{example}

\begin{example}\label{ex:whitehead}
Let $\OO_W$ be the link orbifold of Figure \ref{fig:whitehead}, with base space $S^3$ and singular locus the Whitehead link, labeled $3$. Using Tube, as in Example \ref{ex:52-knot}, we verify that the collar radius of $\Sigma_{\OO_W}$ is $0.27702...$. Lifting to $\HH^3$, we obtain a pair of elliptic axes $\alpha, \gamma$ whose distance is  $2 \times 0.27702...$ from each other. Thus, by Theorem \ref{thm:gmmr}, the $2$--generator group $G_{3,10}$ has an index--$2$ subgroup $H_{3,10} \subset \pi_1(\OO_W)$. As in Example  \ref{ex:52-knot}, volume considerations tell us that $H_{3,10} = \pi_1(\OO_W)$.
\end{example}

Using these tools, we may now complete the proof of Theorem \ref{smalltube}.

\begin{proof}[Proof of Theorem \ref{smalltube}]
Let $\OO$ be a hyperbolic link orbifold with $3$--torsion only, and assume that the collar radius of $\Sigma_\OO$ is $r \leq 0.294$. Let $G = \langle f,g \rangle$ be the two-generator Kleinian group associated to $\Gamma = \pi_1(\OO)$, as in Lemma \ref{gavens-lemma}, and let $H = \langle f, \, gfg \rangle$ be the index--$2$ subgroup of $G$ that is also a subgroup of $\Gamma$.

By Theorem \ref{thm:gmmr}, $G$ is one of the three groups $G_{3,i}$ for $i = 6,7, 10$,  enumerated in Table \ref{table:gmmr} and \cite[Table 6]{GMMR}. Then, by Examples \ref{ex:52-knot} and \ref{ex:whitehead}, we know that $H = H_{3,i}$ is either $\pi_1(\OO_K)$ or $\pi_1(\OO_W)$, where $\OO_K$ and $\OO_W$ are the link orbifolds of Figure \ref{fig:knots} and \ref{fig:whitehead}, respectively.

In other words, $\OO$ is covered by either $\OO_K$ or $\OO_W$. We want to show that the degree of the cover is $1$, i.e.\ $\OO$ is isomorphic to either $\OO_K$ or $\OO_W$.

If the covering degree is $2$, then the cover is regular, given by a symmetry of $\OO_K$ or $\OO_W$. An examination of their symmetry groups, using Snap, confirms that none of their twofold quotients are link orbifolds with $3$--torsion only.

If the covering degree is $3$ or higher, we argue by drilling and filling, as in Section \ref{s:drill-fill}.

Suppose that $\OO$ is covered by $\OO_K$, with covering degree at least $3$. By hypothesis, the collar radius of $\OO$ is $r \geq 0.2448$. Let $\MO = X_\OO \setminus \Sigma_\OO$ be the cusped manifold created by drilling out $\Sigma_\OO$.
By Proposition \ref{drillbound},
$$\vol(\MO) \: \leq \: (\coth^3{2r}) \left( 1 + \frac{0.91}{\cosh{2r}} \right) \cdot \frac{\vol(\OO_K)}{3} \: \leq \:  19.365 \cdot 0.1048  \: = \: 2.029... $$
By Theorem \ref{thm:moms}, $\MO$ is one of the ten manifolds enumerated in that theorem (in fact, its volume is so low that it must be either ${\tt m003}$ or ${\tt m004}$). By Proposition \ref{prop:moser-search}, $\vol(\OO) \geq \vol(\OO_K)$, which contradicts the assumption that it is covered by $\OO_K$ with degree at least $3$.

Similarly, suppose that $\OO$ is covered by $\OO_W$, with covering degree at least $3$. By hypothesis, the collar radius of $\OO$ is $r \geq 0.2770$. Again, let $\MO = X_\OO \setminus \Sigma_\OO$ be the cusped manifold created by drilling out $\Sigma_\OO$.
By Proposition \ref{drillbound},
$$\vol(\MO) \: \leq \: (\coth^3{2r}) \left( 1 + \frac{0.91}{\cosh{2r}} \right) \cdot \frac{\vol(\OO_W)}{3} \: \leq \:  14.00 \cdot 0.1760  \: = \: 2.462... $$
Thus, once again, $\MO$ is one of the ten manifolds enumerated in Theorem \ref{thm:moms},  and  Proposition \ref{prop:moser-search} implies that $\vol(\OO) \geq 0.3142$. But if $\OO$ is covered by $\OO_W$ with degree at least $3$, its volume is at most $0.1760$, which is a contradiction.
\end{proof}

\bibliographystyle{plain}
\bibliography{./atkinson-futer_references}
\end{document}